\documentclass[letterpaper, 10 pt, conference]{ieeeconf}  

\IEEEoverridecommandlockouts                              
\usepackage{graphics} 
\usepackage{epsfig} 
\usepackage{mathptmx} 
\usepackage{times} 
\usepackage{amsmath} 
\usepackage{amssymb}  

\usepackage{mathtools}

\usepackage{algorithm}
\usepackage{algorithmic}
\usepackage{xcolor}
\usepackage{hyperref}

\usepackage{color}

\usepackage{blindtext}
\usepackage{enumerate}
\usepackage{graphicx}
\usepackage{amsfonts, amsmath, bm, amssymb}
\usepackage{dsfont}
\usepackage{wrapfig}
\usepackage{subcaption}

\newcommand{\fracpartial}[2]{\frac{\partial #1}{\partial  #2}}

\usepackage{pifont}
\usepackage{xspace}
\makeatletter
\DeclareRobustCommand\onedot{\futurelet\@let@token\@onedot}
\def\@onedot{\ifx\@let@token.\else.\null\fi\xspace}

\makeatother

\newcommand{\Fc}{\mathcal{F}}
\newcommand{\Gc}{\mathcal{G}}

\newcommand{\Kc}{\mathcal{K}}

\newcommand{\Nc}{\mathcal{N}}

\newcommand{\Wc}{\mathcal{W}}

\newcommand{\Rb}{\mathbb{R}}

\newcommand{\Tb}{\mathbb{T}}

\ifx\BlackBox\undefined
\newcommand{\BlackBox}{\rule{1.5ex}{1.5ex}}  
\fi
\ifx\QED\undefined
\def\QED{~\rule[-1pt]{5pt}{5pt}\par\medskip}
\fi
\ifx\proof\undefined
\newenvironment{proof}{\par\noindent{\em Proof:\ }}{\hfill\BlackBox\\}
\fi
\ifx\theorem\undefined
\newtheorem{theorem}{Theorem}
\fi
\ifx\example\undefined
\newtheorem{example}{Example}
\fi
\ifx\property\undefined
\newtheorem{property}{Property}
\fi
\ifx\lemma\undefined
\newtheorem{lemma}{Lemma}
\fi
\ifx\proposition\undefined
\newtheorem{proposition}{Proposition}
\fi
\ifx\fact\undefined
\newtheorem{fact}{Fact}
\fi
\ifx\remark\undefined
\newtheorem{remark}{Remark}
\fi
\ifx\corollary\undefined
\newtheorem{corollary}{Corollary}
\fi
\ifx\definition\undefined
\newtheorem{definition}{Definition}
\fi
\ifx\conjecture\undefined
\newtheorem{conjecture}{Conjecture}
\fi
\ifx\axiom\undefined
\newtheorem{axiom}[theorem]{Axiom}
\fi
\ifx\claim\undefined
\newtheorem{claim}[theorem]{Claim}
\fi
\ifx\assumption\undefined
\newtheorem{assumption}{Assumption}
\fi
\ifx\question\undefined
\newtheorem{question}{Question}
\fi
\ifx\problem\undefined
\newtheorem{problem}{Problem}
\fi

\title{\LARGE \bf
Learning Lyapunov Operators for Nonlinear Systems
}

\author{Amartya Mukherjee, Maxwell Fitzsimmons, David C. Del Rey Fern\'andez, and Jun Liu 
\thanks{This work was supported in part by the Natural Sciences and Engineering Research Council of Canada and the Canada Research Chairs Program.}
\thanks{Amartya Mukherjee, Maxwell Fitzsimmons, David C. Del Rey Fern\'andez, and Jun Liu are with the Department of Applied Mathematics, University of Waterloo, Waterloo, Ontario, Canada N2L 3G1 (email: {\tt\small j.liu@uwaterloo.ca (Jun Liu)}).}%
}

\begin{document}

\maketitle
\thispagestyle{empty}
\pagestyle{empty}

\begin{abstract}
Constructing Lyapunov functions for nonlinear dynamical systems is a central problem in stability analysis, yet remains challenging.
Lyapunov functions are commonly characterized as solutions to first-order partial differential equations (PDEs), but these solutions are typically obtained for single systems, limiting their reuse across systems.
In this paper, we study the Lyapunov solution operator that maps a vector field to the corresponding Lyapunov function defined by a dissipation-based Lyapunov PDE.
We establish that, on compact subsets of the domain of attraction and under exponential stability assumptions, this operator is well-defined, unique, and continuous with respect to perturbations of both the vector field and the dissipation function.
These results provide a theoretical foundation for approximating Lyapunov functions uniformly over families of nonlinear systems.
Building on these theoretical foundations, we employ Fourier Neural Operators (FNOs) as a data-driven approximation of the Lyapunov solution operator.
Numerical experiments demonstrate that a single trained operator can accurately approximate the numerical Lyapunov functions across parameterized families of dynamics.
This illustrates the potential of neural operators for approximating Lyapunov functions.
\end{abstract}

\section{INTRODUCTION}

One of the longstanding challenges in nonlinear systems and control is the construction of Lyapunov functions \cite{khalil2002nonlinear}. While Lyapunov functions can essentially be characterized by solutions to partial differential equations (PDEs) and neural network solutions to such PDEs can effectively provide approximations to Lyapunov functions \cite{liu2025physics}, solving PDEs for each system can still be time-consuming. Additionally, prior works that use neural networks to learn Lyapunov functions only aid in the verification of a single system \cite{mengphysics}, thus failing to generalize into systems with slightly different dynamics. This could pose difficulties in real-life systems.

Recently, operator learning has emerged as a paradigm for approximating mappings between function spaces, such as those defined by PDEs. Unlike conventional deep learning architectures that learn finite-dimensional mappings, neural operators generalize across function spaces, allowing them to learn solution operators of PDEs from data. This perspective is particularly attractive for Lyapunov analysis: the correspondence between vector fields and their associated Lyapunov functions can be viewed as an operator defined by a PDE constraint. Approximating this operator directly offers the promise of learning a single model that can compute Lyapunov functions for a large class of systems.

The Fourier Neural Operator (FNO) is an example of a neural operator \cite{kossaifi2024neural, kovachki2021neural}. It lifts the input functions to a higher-dimensional feature space using a linear layer, followed by repeated applications of Fourier convolution layers. Each layer performs a fast Fourier transform (FFT), applies learnable filters in the frequency domain, and then inverts the transform to return to the spatial domain. This global convolution mechanism enables FNO to efficiently capture long-range dependencies in the input.

In this paper, we bridge the perspectives of Lyapunov stability and operator learning. We begin by formulating the Lyapunov stability condition as a PDE, thereby recasting the problem into the operator learning framework. We then establish regularity and continuity results for this PDE, showing that the assumptions required for the FNO universality theorem \cite{kovachki2021universal} hold in our setting. This provides theoretical justification for approximating Lyapunov operators using FNOs. We finally train an FNO on nonlinear systems and demonstrate that it can generate Lyapunov functions with small error with respect to the true Lyapunov function.

\section{Problem Formulation}

We consider autonomous nonlinear dynamical systems of the form
\begin{equation}
\dot{x}(t) = f(x(t)), \quad x \in \Rb^n,
\end{equation}
where the vector field $f : \Rb^n \to \Rb^n$ is continuously differentiable and satisfies $f(0)=0$. Let $\phi_f(t,x)$ denote the flow of the system initialized at $x$ at time $t=0$.

\subsection{Admissible class of dynamics}

Define the function space
\[
\Fc_0 := \{ f \in C^1(\Rb^n;\Rb^n) \mid f(0)=0 \},
\]
and the subclass of locally exponentially stable vector fields
\[
\Fc_{{st}} := \left\{ f \in \Fc_0 \;\middle|\; \frac{\partial f}{\partial x}(0) \text{ is Hurwitz} \right\}.
\]
For $f \in \Fc_{{st}}$, classical results guarantee that the origin is a locally exponentially stable equilibrium and that solutions exist and are unique in a neighborhood of the origin.

The \emph{domain of attraction} of the origin is defined as
\begin{equation}
{DOA}(f) := \{ x \in \Rb^n \mid \lim_{t\to\infty} \phi_f(t,x) = 0 \}.
\end{equation}
A set $\Omega$ is called positively invariant if $\phi_f(t,x)\in\Omega$ for all $t\geq 0$ and $x\in\Omega$.

\subsection{Dissipation functions and Lyapunov PDE}

Let $\omega : \Rb^n \to \Rb$ be a twice continuously differentiable function satisfying
\begin{equation}\label{eq:omega}
\omega(0)=0, \quad \nabla \omega(0)=0, \quad \nabla^2\omega \succ 0,
\end{equation}
where $\nabla^2\omega$ denotes the Hessian matrix. Define the admissible class
\[
\Wc_{>0} := \left\{ \omega \in C^2(\Rb^n,\Rb) \;\middle|\; \omega \text{ satisfies conditions \eqref{eq:omega}} \right\}.
\]

Given $f \in \Fc_{{st}}$ and $\omega \in \Wc_{>0}$, we consider Lyapunov functions defined as solutions of the first-order PDE
\begin{equation}
\nabla V(x) \cdot f(x) = -\omega(x), \quad V(0)=0.
\label{eq:lyap-pde}
\end{equation}
Equations of this form arise naturally in converse Lyapunov theory and characterize dissipation-based Lyapunov functions.

In applications, local Lyapunov functions are found to prove local asymptotic/exponential stability, but they are also used to estimate the size of $DOA(f)$. This is because if $V$ is a Lyapunov function for $f$ on $K=\{x\in\Rb^n:V(x)<r\}$ for a fixed $r$, then $K$ is positively invariant and $K\subset DOA(f)$. Moreover, there is a Lyapunov function $V_1$ for $f$ such that $\{x\in\Rb^n:V_1(x)<r\}=DOA(f)$ \cite{vannelli1985maximal,zubov1964methods}.

\subsection{Integral representation and well-posedness}

Let $K \subset {DOA}(f)$ be compact and positively invariant. Under this assumption, the Lyapunov PDE~\eqref{eq:lyap-pde} admits the integral representation
\begin{equation}
V_{f,\omega}(x) := \int_0^\infty \omega\bigl(\phi_f(t,x)\bigr)\,dt, \quad x \in K,
\label{eq:lyap-integral}
\end{equation}
which is well defined due to the exponential decay of trajectories and the regularity of $\omega$.

\begin{proposition}[Existence and uniqueness]\label{prop:existence_uniqueness}
Let $f \in \Fc_{{st}}$, $\omega \in \Wc_{>0}$, and let $K \subset {DOA}(f)$ be compact and positively invariant. Then the following are equivalent:
\begin{enumerate}
    \item There exists a unique function $V \in C^1(K)$ satisfying~\eqref{eq:lyap-pde}
    \item This solution is given by the integral formula~\eqref{eq:lyap-integral}.
\end{enumerate}
\end{proposition}

\begin{proof}
If (1) holds, then we see that \[\frac{d}{dt}(V_{f,\omega}\circ\phi_f(t,x))=-\omega\circ\phi_f(t,x)\] for $t\geq 0$ and $x\in K$. Integrating from $0$ to $T$, we find
\[V_{f,\omega}\circ\phi_f(T,x)-V_{f,\omega}(x)=\int_0^T-\omega\circ\phi_f(s,x)ds.\]
We note that, as $\omega$ is Lipschitz, the integrated solution is unique.
Taking the limit as $T\to\infty$ yields the result, after noting that $\phi_f(T,x)\to 0$ if $x\in K$ and $V(0)=0$.

If (2) holds then $V_{f,\omega}\circ\phi_f(T,x)-V_{f,\omega}(x)=\int_0^T-\omega\circ\phi_f(s,x)ds.$ By the mean value theorem for integrals, we obtain
\[V_{f,\omega}\circ\phi_f(T,x)-V_{f,\omega}(x)=-T\omega\circ\phi_f(s',x)\]
for $s'\in[0,T]$. Dividing both sides by $T$ and taking the limit $T\to 0$ yields the result.
Finally, $V$ is continuously differentiable on $K$ as proved by \cite{liu2025physics}.
\end{proof}

As a consequence, $V_{f,\omega}$ is a Lyapunov function for $f$ on $K$, and its sublevel sets define invariant subsets of the domain of attraction.

\subsection{The Lyapunov solution operator}

Rather than constructing Lyapunov functions on a per-system basis, we adopt an operator-theoretic viewpoint. Define the \emph{Lyapunov solution operator}
\[
\Gc:\Fc_{st}\times\Wc_{>0},\quad
\Gc(f,\omega) := V_{f,\omega},
\]
mapping a vector field and dissipation function to the corresponding Lyapunov function. 
The central objective of this work is to establish continuity of this operator on compact subsets of the domain of attraction.

\subsection{Sobolev regularity of the data and solution}\label{sec:sobolev}

The universality results for FNOs are formulated for operators acting between Sobolev spaces, while Proposition \ref{prop:existence_uniqueness} defines solutions to the Lyapunov PDE in the $C^1$ and $C^2$ space.
To place the Lyapunov solution operator within this framework, it is necessary to ensure that both the input and output functions admit sufficient Sobolev regularity. 
The Sobolev regularity results here are standard in the analysis of PDEs and play a central role in neural operator theory \cite{evans2022partial}.
They ensure that the Lyapunov PDE \eqref{eq:lyap-pde} is well defined pointwise while simultaneously allowing us to work in function spaces compatible with neural operator approximation.
The next two theorems ensure that the solutions to the Lyapunov PDE can be embedded in appropriate Sobolev spaces. 

\begin{theorem}[Sobolev embedding \cite{Sob38}]\label{thm:sob}
    Let $\Omega \subset \Rb^n$ be a bounded open domain with Lipschitz boundary, and let $k \ge 0$ be an integer. If $s > n/2 + k$, then the Sobolev space $H^s(\Omega)$ is continuously and compactly embedded in $C^k(\overline{\Omega})$, i.e., 
    $H^s(\Omega) \hookrightarrow C^k(\overline{\Omega})$.
    Moreover, there exists a constant $C > 0$ such that
    \begin{equation*}
        \|u\|_{C^k(\overline{\Omega})} \leq C \|u\|_{H^s(\Omega)},\quad \forall u \in H^s(\Omega).
    \end{equation*}
\end{theorem}

\begin{corollary}[Sobolev embedding for vector-valued functions]\label{cor:sob}
    Let $\Omega \subset \Rb^n$ be a bounded open domain with Lipschitz boundary, and let $k \ge 0$ be an integer. If $s > n/2 + k$, then the Sobolev space $H^s(\Omega;\Rb^n)$ consisting of $\Rb^n$-valued functions with $s$-Sobolev regularity is continuously and compactly embedded in $C^k(\overline{\Omega};\Rb^n)$, i.e.,
    $H^s(\Omega;\Rb^n) \hookrightarrow C^k(\overline{\Omega};\Rb^n)$.
    Moreover, there exists a constant $C > 0$ such that
    \begin{equation*}
        \|u\|_{C^k(\overline{\Omega};\Rb^n)} \leq C \|u\|_{H^s(\Omega;\Rb^n)},\quad \forall u \in H^s(\Omega;\Rb^n).
    \end{equation*}
\end{corollary}

The proof of Corollary \ref{cor:sob} follows directly from the scalar case by applying the component-wise argument.

\section{Continuity and Universality of the Lyapunov Solution Operator}

In this section, we establish the main theoretical result of the paper: continuity of the Lyapunov solution operator with respect to perturbations of the vector field and the dissipation function. This property is essential for approximating Lyapunov functions uniformly over families of nonlinear systems.

\subsection{FNO universality theorem}

The work of \cite{kovachki2021universal} provides the key conditions imposed on the input functions and output function of an operator so that the FNO universality theorem applies. In our setting, the operator is the Lyapunov operator, which parametrizes the Lyapunov PDE over the vector field $f$ and the dissipation function $\omega$. To apply the universality result, we must verify the following conditions:
\begin{enumerate}
    \item the input vector fields belong to a Sobolev space $H^s(\Omega;\Rb^n)$ and the dissipation functions belong to a Sobolev space $H^r(\Omega)$ with sufficiently high regularity;
    \item the corresponding Lyapunov functions belong to a Sobolev space $H^{s'}(\Omega)$;
    \item the Lyapunov solution operator is continuous on compact subsets of the product of the input spaces.
\end{enumerate}
Under these assumptions, the modified universality theorem of \cite{kovachki2021universal} for the Lyapunov PDE yields convergent approximation by FNOs. We restrict attention to compact subsets of the admissible classes $\Fc_{st}$ and $\Wc_{>0}$, viewed as subsets of appropriate Sobolev spaces via Sobolev embedding.

\begin{theorem}[Modification of Theorem 9 by \cite{kovachki2021universal}]
\label{thm:fno_lyap_universality}
Let $\Omega \subset \Rb^n$ be a bounded domain with Lipschitz boundary such that $\overline{\Omega}\subset (0,2\pi)^n$. Let $s,r,s' \ge 0$, and let
\[
\Kc_f \subset \Fc_{\mathrm{st}} \cap H^s(\Omega;\Rb^n), 
\quad 
\Kc_\omega \subset \Wc_{>0} \cap H^r(\Omega)
\]
be compact sets of admissible vector fields and dissipation functions, respectively. Assume that for each $(f,\omega) \in \Kc_f \times \Kc_\omega$, the Lyapunov PDE
\begin{equation}
\nabla V(x)\cdot f(x) = -\omega(x), \quad x\in \Omega,
\label{eq:PDE}
\end{equation}
admits a unique solution $V_{f,\omega} \in H^{s'}(\Omega)$, and that the associated Lyapunov solution operator
\[
\Gc : \Kc_f \times \Kc_\omega \to H^{s'}(\Omega), 
\quad \Gc(f,\omega)=V_{f,\omega},
\]
is continuous with respect to the product topology induced by the $H^s(\Omega;\Rb^n)$ and $H^r(\Omega)$ norms on the input and the $H^{s'}(\Omega)$ norm on the output.
Let $\Tb^n$ denote the $n$-dimensional torus.
Then, for every $\varepsilon>0$, there exist
\begin{enumerate}
    \item continuous linear extension operators
    \[
    E_f : H^s(\Omega;\Rb^n) \to H^s(\Tb^n;\Rb^n),\quad
    E_\omega : H^r(\Omega) \to H^r(\Tb^n),
    \]
    \item a Fourier Neural Operator
    \[
    \Nc_\varepsilon : H^s(\Tb^n;\Rb^n) \times H^r(\Tb^n) \to H^{s'}(\Tb^n),
    \]
\end{enumerate}
such that
\[
\sup_{(f,\omega)\in \Kc_f \times \Kc_\omega}
\left\|
\Gc(f,\omega)-\Nc_\varepsilon(E_f f, E_\omega \omega)\big|_\Omega
\right\|_{H^{s'}(\Omega)}
<\varepsilon.
\]
In particular, the Lyapunov solution operator can be approximated arbitrarily well on compact subsets of admissible vector fields and dissipation functions by a suitable FNO.
\end{theorem}

The regularity conditions required for this theorem are satisfied in our setting. Specifically, from the Sobolev embedding results in Section \ref{sec:sobolev}, for $s > n/2 + 1$ we have $H^s(\Omega;\Rb^n) \hookrightarrow C^1(\overline{\Omega};\Rb^n)$, ensuring that vector fields are continuously differentiable. For the dissipation functions, we require $r > n/2 + 2$ to guarantee $H^r(\Omega) \hookrightarrow C^2(\overline{\Omega})$, which is the natural regularity for $\omega$ in the Lyapunov PDE. The continuity of the Lyapunov solution operator $\Gc$ is established in Theorem \ref{thm:continuity}, providing the final ingredient needed to apply the universality result.

\subsection{Regularity of the Lyapunov PDE}

Before turning to continuity, we briefly comment on regularity.  
Under Section \ref{sec:sobolev}, the Lyapunov PDE
admits a unique solution $V_{f,\omega} \in C^1(J)$ on compact invariant sets $J \subset {DOA}(f)$. In numerical settings, vector fields and Lyapunov functions are often represented in Sobolev spaces. For sufficiently large $s > n/2+1$, classical Sobolev embedding results ensure continuous and compact embeddings
\[
H^s(J;\Rb^n) \hookrightarrow C^1(J;\Rb^n),
\quad
H^s(J) \hookrightarrow C^1(J),
\]
which guarantee that the Lyapunov PDE is well defined in the function spaces used by neural operator architectures.

\subsection{Main continuity result}

We begin by stating the central theorem.

\begin{theorem}[Continuity of the Lyapunov solution operator]
\label{thm:continuity}
Let $f \in \Fc_{{st}}$ and $\omega \in \Wc_{>0}$.  
Let $K \subset {DOA}(f)$ be compact. Then there exists a compact set,
\[
K \subset J \subset {DOA}(f),
\]
such that the Lyapunov solution operator
\[
\Gc : \Fc_{{st}} \times \Wc_{>0} \to C^1(J), 
\quad \Gc(g,\psi) := V_{g,\psi},
\]
is continuous at $(f,\omega)$ when $\Fc_{{st}} \times \Wc_{>0}$ is endowed with the norm
\[
\|g\|_{C^1(J)} + \|\psi\|_{C^2(J)},
\]
and $C^1(J)$ is endowed with the uniform norm $\|\cdot\|_{\infty,J}$.
\end{theorem}

\begin{remark}
Theorem~\ref{thm:continuity} formalizes the intuition that small perturbations of the system dynamics and dissipation function induce small changes in the associated Lyapunov function on compact invariant sets. This robustness property provides the theoretical foundation for learning Lyapunov functions uniformly over families of nonlinear systems.
\end{remark}

The proof of Theorem~\ref{thm:continuity} proceeds in three steps. First, we establish robustness of exponential stability under perturbations of the vector field. Second, we show uniform decay of the Lyapunov function outside small sublevel sets. Finally, we prove uniform integrability of the Lyapunov integral representation.

\subsection{Robust exponential stability}

We now establish robustness of local exponential stability with respect to perturbations of the vector field.

\begin{lemma}[Robust exponential stability]
\label{lem:robust-stability}
Let $f \in \Fc_{{st}},\omega\in\Wc_{>0}$. There exist constants $r>0$, $\delta>0$, $M>0$, and $c>0$ such that for any
\[
g \in \Fc_0 \quad \text{with} \quad \|g-f\|_{C^1(J)} < r,
\]
the origin remains locally exponentially stable for $\dot{x}=g(x)$ and
\[
\|\phi_g(t,x)\| \le M e^{-ct}\|x\|, \quad \forall\, t\ge 0,\ \|x\|\le \delta.
\]
\end{lemma}

\begin{proof}
Define the ball $B(x,\eta):=\{y\in\Rb^n:\|x-y\|<\eta\}$, and define $\overline{B}(x,\eta)$ as its closure.
Choose $\eta>0$ such that $B(0,\eta)\subset DOA(f)$. Let $J=K\cup\overline{B}(0,\eta)$.
Let $A:=\fracpartial{f}{x}(0),Q=\nabla^2\omega(0)$. Since $f\in\Fc_{st}$, $A$ is Hurwitz, and since $\omega\in\Wc_{>0}$, $Q\succ 0$.
Let $P\in\Rb^{n\times n}$ be the positive definite matrix that satisfies \cite{khalil2002nonlinear}
\[PA+A^TP=-Q.\]
Let $\lambda_{\min}(\cdot)$ and $\lambda_{\max}(\cdot)$ refer to the smallest and largest eigenvalues respectively.
Define $W(x):=\frac{1}{2}x^TPx$, then $W$ is a local Lyapunov function for $f$. Let
\[c_1:=\frac{\lambda_{\min}(Q)}{6\lambda_{\max}(P)}.\]
We will show that, for sufficiently small $\delta>0$ and $r>0$, the function $W$ is a strict local Lyapunov function for the system $\dot x=g(x)$ whenever $\|g-f\|_{C^1(J)}<r$.

Since $f\in C^1$, there exists $\delta_1>0$ such that $\delta_1\le \eta$ and
\[
\|\xi\|\le \delta_1\quad\implies\quad
\left\|\frac{\partial f}{\partial x}(\xi)-A\right\|\le c_1.
\]
Choose $r:=c_1$.
Now let $g\in \Fc_0$ satisfy $\|g-f\|_{C^1(J)}<r$, and let $\|x\|\le \delta_1$.
Since $g(0)=0$, Taylor's theorem gives
\[
g(x)=\int_0^1 \frac{\partial g}{\partial x} (tx) dt\, x 
\]
Hence,
\begin{align*}
\frac{d}{dt}W(\phi_g(t,x))\Big|_{t=0}
=~& x^TPg(x) \\
=~& x^TP\int_0^1 \frac{\partial g}{\partial x} (tx) dt\,x \\
=~& x^TPAx
   +x^TP\int_0^1\left(\frac{\partial g}{\partial x}(tx)-\frac{\partial f}{\partial x}(tx)\right)dt\,x\\
   &+x^TP\int_0^1\left(\frac{\partial f}{\partial x}(tx)-A\right)dt\,x.
\end{align*}
Since $x^TPAx=\frac12 x^T(PA+A^TP)x=-\frac12 x^TQx$, we obtain
\begin{align*}
&\frac{d}{dt}W(\phi_g(t,x))\Big|_{t=0}\\
\le~& -\frac12 \lambda_{\min}(Q)\|x\|^2
+\|P\|\left\|\frac{\partial g}{\partial x}-\frac{\partial f}{\partial x}\right\|_\infty\|x\|^2 \\
&+\|P\|\left\|\frac{\partial f}{\partial x}-A\right\|_\infty\|x\|^2 \\
\le~& -\frac12 \lambda_{\min}(Q)\|x\|^2
   +\|P\|r\|x\|^2
+\|P\|\frac{\lambda_{\min}(Q)}{6\|P\|}\|x\|^2 \\
\le~& -\frac12 \lambda_{\min}(Q)\|x\|^2
   +\frac16\lambda_{\min}(Q)\|x\|^2
+\frac16\lambda_{\min}(Q)\|x\|^2 \\
=~& -\frac16 \lambda_{\min}(Q)\|x\|^2.
\end{align*}
Using $2W(x)\le \lambda_{\max}(P)\|x\|^2$, it follows that
\[
\frac{d}{dt}W(\phi_g(t,x))\Big|_{t=0}
\le -\frac{\lambda_{\min}(Q)}{3\lambda_{\max}(P)}\,W(x)
= -2c\,W(x)
\]
for all $\|x\|\le \delta_1$.

Therefore, along any trajectory of $\dot x=g(x)$ that remains in $B(0,\delta_1)$,
\[
\frac{d}{dt}W(\phi_g(t,x)) \le -2c\,W(\phi_g(t,x)).
\]
By Gr\"onwall's inequality,
\[
W(\phi_g(t,x))\le e^{-2ct}W(x), \quad t\ge 0.
\]
In particular, since $W$ is decreasing, the sublevel set
\[
\Omega_\delta:=\{x\in\Rb^n: W(x)\le \tfrac12 \lambda_{\min}(P)\delta^2\}
\]
is positively invariant whenever $\delta\le \delta_1$. Choose $\delta=\delta_1.$
Then every trajectory with $\|x\|\le \delta$ remains in $B(0,\delta)\subset B(0,\eta)\subset J$ for all $t\ge 0$, so the above estimate is valid globally in time.

Finally, using the bounds relating $W$ and $\|x\|^2$,
\begin{align*}
\|\phi_g(t,x)\|^2
&\le \frac{2}{\lambda_{\min}(P)}W(\phi_g(t,x)) \\
&\le \frac{2}{\lambda_{\min}(P)}e^{-2ct}W(x) \\
&\le \frac{\lambda_{\max}(P)}{\lambda_{\min}(P)}e^{-2ct}\|x\|^2.
\end{align*}
Hence,
\[
\|\phi_g(t,x)\|
\le \sqrt{\frac{\lambda_{\max}(P)}{\lambda_{\min}(P)}}\,e^{-ct}\|x\|.
\]
Setting
\[
M:=\sqrt{\frac{\lambda_{\max}(P)}{\lambda_{\min}(P)}},
\]
we obtain
\[
\|\phi_g(t,x)\| \le M e^{-ct}\|x\|, \quad \forall\, t\ge 0,\ \|x\|\le \delta.
\]
This proves local exponential stability of the origin for $\dot x=g(x)$.
\end{proof}

\subsection{Uniform decay of the Lyapunov derivative}

We next show that the Lyapunov function $V_{f,\omega}$ decreases uniformly along trajectories outside small sublevel sets.

\begin{lemma}[Uniform Lyapunov decay]
\label{lem:uniform-decay}
Let $f \in \Fc_{st}$, $\omega \in \Wc_{>0}$, and let $J \subset DOA(f)$ be compact. Define
$U(\delta) := \{x \in \Rb^n : V_{f,\omega}(x) < \delta\}.$
Then there exists $\delta_0>0$ such that for every $\delta \in (0,\delta_0)$, there exist constants $c_\delta>0$ and $r_\delta>0$ such that for all
\[
g \in \Fc_0 \quad \text{with} \quad \|g-f\|_{C^1(J)} < r_\delta,
\]
the inequality
\[
\nabla V_{f,\omega}(x)\cdot g(x) \le -c_\delta
\]
holds for all $x \in J \setminus U(\delta)$.
\end{lemma}

\begin{proof}
Let $Q := \nabla^2\omega(0) \succ 0$, and denote $\lambda_{\min}(Q)>0$ its smallest eigenvalue.

Let
\[
R := \max_{x \in J} V_{f,\omega}(x), 
\quad
K_2 := \{x \in \Rb^n : V_{f,\omega}(x) \le R\}.
\]
Then $K_2$ is compact, positively invariant under $f$, and satisfies
\[
J \subseteq K_2 \subseteq DOA(f).
\]

We first establish a lower bound on $\omega$ near the origin.
Since $\omega \in C^2$ and $\nabla^2\omega(0)=Q\succ 0$, there exists $\delta_1>0$ such that for all $\|x\|\le \delta_1$,
\[
\|\nabla^2\omega(x)-Q\| \le \frac{\lambda_{\min}(Q)}{2}.
\]
By Taylor's theorem, for such $x$,
\[
\omega(x) = \frac12 x^T \nabla^2\omega(\xi) x
\]
for some $\xi$ on the segment between $0$ and $x$. Hence,
\begin{align}
\omega(x)
&= x^T Q x + x^T \bigl(\nabla^2\omega(\xi)-Q\bigr)x \nonumber \\
&\ge \lambda_{\min}(Q)\|x\|^2 - \|\nabla^2\omega(\xi)-Q\|\|x\|^2 \nonumber \\
&\ge \frac{\lambda_{\min}(Q)}{2}\|x\|^2.
\label{eq:omega-lower}
\end{align}


Since $V_{f,\omega}$ is continuous and $V_{f,\omega}(0)=0$, there exists $\delta_0>0$ such that for all $\delta\in(0,\delta_0)$,
\[
U(\delta) \subset B(0,\delta_1).
\]
Fix such a $\delta$, then for $x \in B(0,\delta_1)\setminus U(\delta)$, we have $V_{f,\omega}(x)\ge \delta$.  
By Taylor's theorem, followed by Cauchy-Schwarz inequality,
\[
\delta \le \|\nabla V_{f,\omega}\|_{\infty,K_2}\|x\|
\quad\implies\quad
\|x\| \ge \frac{\delta}{\|\nabla V_{f,\omega}\|_{\infty,K_2}}.
\]
Combining with \eqref{eq:omega-lower}, we obtain
\begin{equation}
\omega(x)
\ge
\frac{\lambda_{\min}(Q)}{2}
\frac{\delta^2}{\|\nabla V_{f,\omega}\|_{\infty,K_2}^2},
\quad x \in B(0,\delta_1)\setminus U(\delta).
\label{eq:omega-inner}
\end{equation}

Next, we establish a global lower bound on $\omega$ outside $U(\delta)$.
To separate our analysis into large and small level sets,
we decompose the set
\[
K_2 \setminus U(\delta)
=
\bigl(K_2 \setminus B(0,\delta_1)\bigr)
\cup
\bigl(B(0,\delta_1)\setminus U(\delta)\bigr).
\]
Since $\omega$ is continuous and strictly positive, the minimum away from the origin,
\[
m_1 := \min_{x \in K_2 \setminus B(0,\delta_1)} \omega(x),
\]
is strictly positive. Combining this with \eqref{eq:omega-inner}, we obtain
\[
\min_{x \in K_2 \setminus U(\delta)} \omega(x)
\ge
\min\left\{
m_1,\;
\frac{\lambda_{\min}(Q)}{2}
\frac{\delta^2}{\|\nabla V_{f,\omega}\|_{\infty,K_2}^2}
\right\}.
\]
Define
\[
c_\delta := \frac12
\min\left\{
m_1,\;
\frac{\lambda_{\min}(Q)}{2}
\frac{\delta^2}{\|\nabla V_{f,\omega}\|_{\infty,K_2}^2}
\right\} > 0.
\]
Then,
\begin{equation}
\min_{x \in K_2 \setminus U(\delta)} \omega(x) \ge 2c_\delta.
\label{eq:omega-global}
\end{equation}

Finally, we prove the main perturbation argument.
Let
\[
r_\delta := \frac{c_\delta}{\|\nabla V_{f,\omega}\|_{\infty,K_2}}.
\]
Let $g \in \Fc_0$ satisfy $\|g-f\|_{C^1(J)} < r_\delta$.  
Then, for all $x \in K_2 \setminus U(\delta)$,
\begin{align*}
\nabla V_{f,\omega}(x)\cdot g(x)
&=
-\omega(x) + \nabla V_{f,\omega}(x)\cdot (g(x)-f(x)) \\
&\le
-2c_\delta
+ \|\nabla V_{f,\omega}\|_{\infty,K_2}\|g-f\|_{\infty,K_2} \\
&<
-2c_\delta + \|\nabla V_{f,\omega}\|_{\infty,K_2} r_\delta \\
&= -c_\delta.
\end{align*}

This proves the result.
\end{proof}

\subsection{Uniform integrability of the Lyapunov integral}

We now establish convergence and robustness of the integral representation.

\begin{lemma}[Uniform integrability]
\label{lem:integrability}
Let $f \in \Fc_{st}$ and $\omega \in \Wc_{>0}$.  
There exist constants $\delta>0$, $r>0$, and $C>0$ such that for all
\[
g \in \Fc_0,\ \psi \in \Wc_{>0}
\quad \text{with} \quad
\|g-f\|_{C^1(J)} + \|\psi-\omega\|_{C^2(J)} < r,
\]
the integral
\[
V_{g,\psi}(x) = \int_0^\infty \psi(\phi_g(t,x))\,dt
\]
is well defined for all $x \in J$, and
\[
\int_{T}^\infty |\psi(\phi_g(t,x))|\,dt \le C\delta
\]
for all $x \in J$ and all sufficiently large $T$.
\end{lemma}

\begin{proof}
Fix $\delta>0$ sufficiently small as in Lemma~\ref{lem:uniform-decay}, and define
\[
U(\delta) := \{x \in \Rb^n : V_{f,\omega}(x) < \delta\}.
\]
Let $x \in K_2$ and define the hitting time
\[
T(\delta,x,g) := \inf \{ t \ge 0 : V_{f,\omega}(\phi_g(t,x)) \le \delta \}.
\]
From Lemma~\ref{lem:uniform-decay}, we have
\[
\frac{d}{dt} V_{f,\omega}(\phi_g(t,x)) \le -c_\delta
\quad \text{for } t \in [0, T(\delta,x,g)].
\]
Integrating and using $V_{f,\omega}(\phi_g(T(\delta,x,g),x)) = \delta$, we obtain
\[
T(\delta,x,g) \le \frac{V_{f,\omega}(x)-\delta}{c_\delta}.
\]
Since $V_{f,\omega}$ is bounded on $K_2$, there exists a constant $T_\delta>0$ such that
\[
T(\delta,x,g) \le T_\delta \quad \text{for all } x \in K_2.
\]
By construction, $U(\delta) \subset B(0,\delta_1)$ for sufficiently small $\delta$.  
From Lemma~\ref{lem:robust-stability}, there exist constants $M_1>0$ and $c_1>0$ such that for all $t \ge T(\delta,x,g)$,
\[
\|\phi_g(t,x)\|
\le
M_1 e^{-c_1 (t - T(\delta,x,g))} \|\phi_g(T(\delta,x,g),x)\|.
\]
Since $\phi_g(T(\delta,x,g),x) \in U(\delta)$, we have
\[
\|\phi_g(T(\delta,x,g),x)\| \le \delta,
\]
and therefore
\begin{equation}
\|\phi_g(t,x)\|
\le
M_1 \delta e^{-c_1 (t - T(\delta,x,g))},
\quad t \ge T(\delta,x,g).
\label{eq:traj-decay}
\end{equation}

Since $\psi$ lies in a bounded subset of $C^2(K_2)$, its gradient is uniformly bounded on $K_2$.  
Hence, there exists a constant $L>0$ such that
\[
|\psi(y)| \le L \|y\|, \quad \forall\, y \in K_2.
\]

Using \eqref{eq:traj-decay}, for $t \ge T(\delta,x,g)$,
\[
|\psi(\phi_g(t,x))|
\le L \|\phi_g(t,x)\|
\le L M_1 \delta e^{-c_1 (t - T(\delta,x,g))}.
\]
Therefore,
\begin{align*}
\int_{T(\delta,x,g)}^\infty |\psi(\phi_g(t,x))|\,dt
&\le
L M_1 \delta \int_0^\infty e^{-c_1 s}\,ds 
=
\frac{L M_1}{c_1} \delta.
\end{align*}

Define $C := \frac{L M_1}{c_1}$. Then for all $x \in K_2$,
\[
\int_{T(\delta,x,g)}^\infty |\psi(\phi_g(t,x))|\,dt \le C\delta.
\]
Since $T(\delta,x,g)$ is uniformly bounded over $K_2$, the integral defining $V_{g,\psi}(x)$ is finite for all $x \in J$.
\end{proof}

\subsection{Proof of Theorem~\ref{thm:continuity}}

\begin{proof}
Let $K \subset DOA(f)$ be compact. Define
\[
R := \max_{x \in K} V_{f,\omega}(x),
\quad
J := \{x \in \Rb^n : V_{f,\omega}(x) \le R\}.
\]
Then $J$ is compact, positively invariant under $f$, and satisfies
\[
K \subset J \subset DOA(f).
\]

Fix $\varepsilon > 0$. We will show that for $(g,\psi)$ sufficiently close to $(f,\omega)$ in $C^1(J)\times C^2(J)$,
\[
\|V_{g,\psi} - V_{f,\omega}\|_{\infty,J} < \varepsilon.
\]

For $x \in J$, we use the integral representation
\[
V_{f,\omega}(x) - V_{g,\psi}(x)
=
\int_0^\infty \bigl[\omega(\phi_f(t,x)) - \psi(\phi_g(t,x))\bigr]\,dt.
\]
We decompose the integrand:
\begin{align*}
|\omega(\phi_f(t,x)) - \psi(\phi_g(t,x))|
&\le
|\omega(\phi_f(t,x)) - \omega(\phi_g(t,x))| \\
&\quad +
|\omega(\phi_g(t,x)) - \psi(\phi_g(t,x))|.
\end{align*}

Fix $\delta>0$ and define
$T_\delta := \sup_{x \in J} T(\delta,x,g)$,
which is finite by Lemma~\ref{lem:integrability}.

For $t \in [0,T_\delta]$, the flows $\phi_f$ and $\phi_g$ remain in $J$.  
By Gr\"onwall inequality, there exists $C_1>0$ such that
\[
\sup_{x \in J,\, t \in [0,T_\delta]}
\|\phi_f(t,x) - \phi_g(t,x)\|
\le
C_1 \|f-g\|_{C^1(J)}.
\]

Since $\omega \in C^1(J)$, it is Lipschitz on $J$, so there exists $L_\omega>0$ such that
\[
|\omega(\phi_f(t,x)) - \omega(\phi_g(t,x))|
\le
L_\omega \|\phi_f(t,x) - \phi_g(t,x)\|.
\]
Hence,
\[
\int_0^{T_\delta} |\omega(\phi_f(t,x)) - \omega(\phi_g(t,x))|\,dt
\le
C_2 \|f-g\|_{C^1(J)}
\]
for some constant $C_2>0$.

Similarly,
\[
\int_0^{T_\delta} |\omega(\phi_g(t,x)) - \psi(\phi_g(t,x))|\,dt
\le
T_\delta \|\omega-\psi\|_{\infty,J}.
\]

For $t \ge T(\delta,x,g)$, Lemma~\ref{lem:integrability} yields
\[
\int_{T(\delta,x,g)}^\infty |\psi(\phi_g(t,x))|\,dt \le C\delta.
\]
Applying the same argument to $(f,\omega)$,
\[
\int_{T(\delta,x,f)}^\infty |\omega(\phi_f(t,x))|\,dt \le C\delta.
\]

Thus, the tail contribution satisfies
\[
\int_{T_\delta}^\infty |\omega(\phi_f(t,x)) - \psi(\phi_g(t,x))|\,dt
\le 2C\delta.
\]

Combining the estimates, we obtain
\[
\|V_{f,\omega} - V_{g,\psi}\|_{\infty,J}
\le
C_2 \|f-g\|_{C^1(J)}
+ T_\delta \|\omega-\psi\|_{\infty,J}
+ 2C\delta.
\]

First choose $\delta>0$ such that $2C\delta < \varepsilon/3$.  
Then choose $(g,\psi)$ sufficiently close to $(f,\omega)$ so that
\[
C_2 \|f-g\|_{C^1(J)} < \varepsilon/3,
\quad
T_\delta \|\omega-\psi\|_{\infty,J} < \varepsilon/3.
\]

It follows that
\[
\|V_{g,\psi} - V_{f,\omega}\|_{\infty,J} < \varepsilon.
\]

Therefore, $\Gc$ is continuous at $(f,\omega)$.
\end{proof}

\section{Numerical Results}

\begin{figure*}
    \centering
    \includegraphics[width=\linewidth]{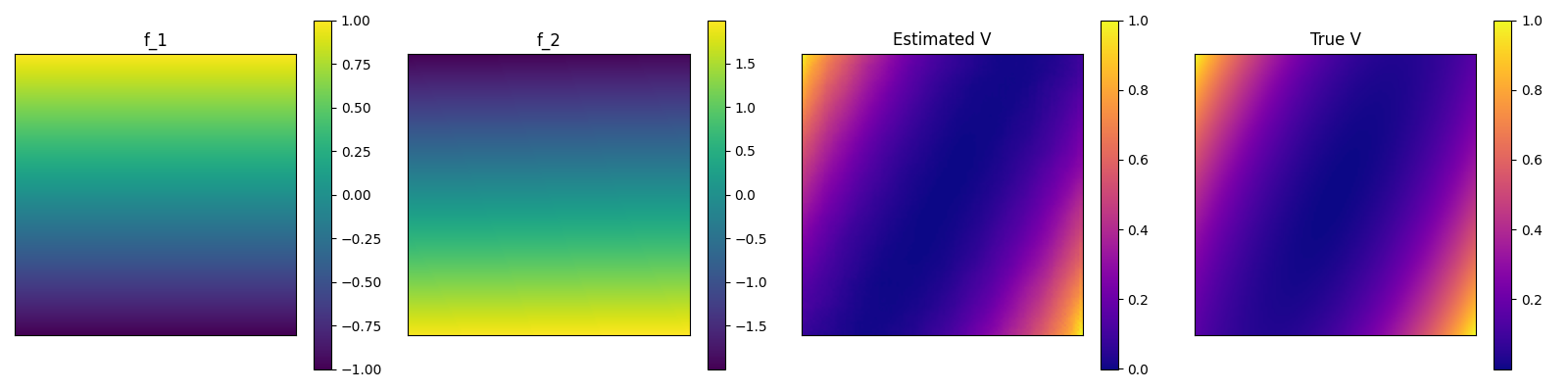}
    \includegraphics[width=\linewidth]{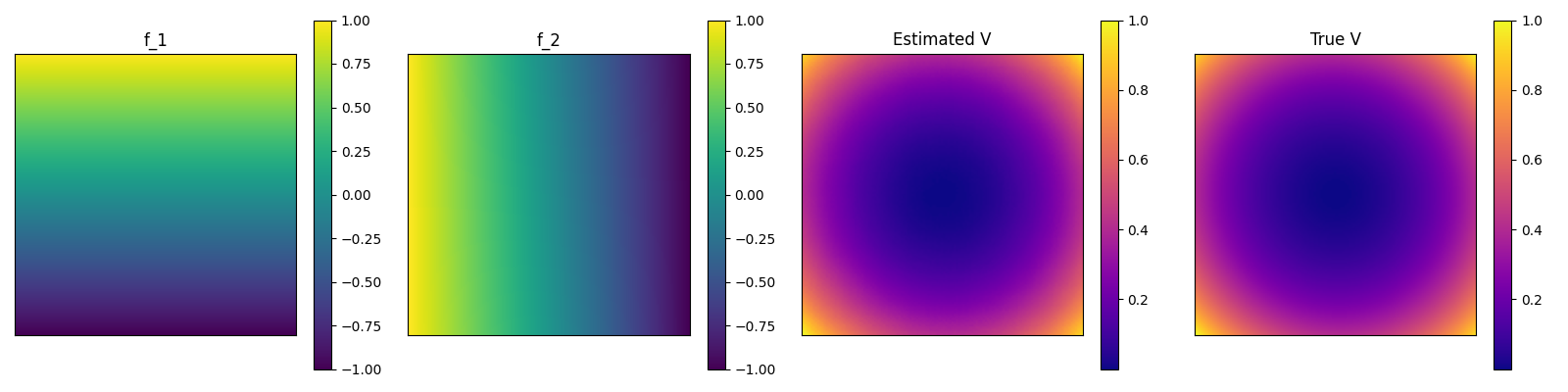}
    \includegraphics[width=\linewidth]{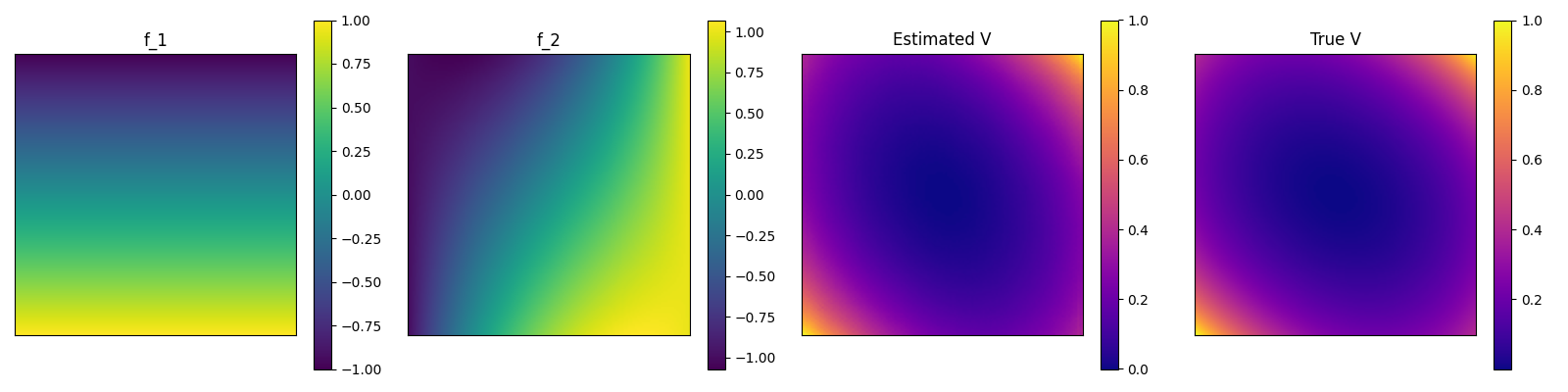}
    \caption{Comparison of learned Lyapunov functions (output $V$) against ground truth solutions (true $V$) for representative test cases across the inverted pendulum (Row 1), damped Duffing oscillator (Row 2), and Van der Pol oscillator (Row 3). Each subplot shows the input vector field components $(f_1,f_2)$ and the corresponding Lyapunov function values. The FNO reconstructions closely match the ground truth across all systems, closely matching the numerical reference solutions.}
    \label{fig:placeholder}
\end{figure*}

\subsection{Dataset Generation}

We constructed datasets for three representative nonlinear dynamical systems: the damped Duffing oscillator, the inverted pendulum, and the Van der Pol oscillator. For each system, 1,000 parameterized instances were generated, with parameters sampled from uniform distributions as specified below. Lyapunov functions were obtained by numerically solving the PDE $\dot V=-x^TQx$ using the LyZNet toolbox \cite{liu2024tool}. $Q$ is a randomly sampled positive definite matrix. Each dataset was divided into 800 training, 100 validation, and 100 test samples. Each system was projected onto the grid $[-1,1]^2$ divided into a 64x64 grid during training of the FNO.

\begin{itemize}
\item Damped Duffing Equation:
\begin{align}
    \dot{x}_1&=x_2,\\
    \dot{x}_2&=-\delta x_2-\alpha x_1-\beta x_1^3,
\end{align}
with parameters sampled as $\alpha\sim U(1,10),\beta\sim U(0.1,2.0),\delta\sim U(0.1,1.0)$.

\item Inverted Pendulum:
\begin{align}
    \dot\theta_1&=\theta_2,\\
    \dot\theta_2&=-c\theta_2-\frac{c}{l}\sin\theta_1,
\end{align}
with parameters $l\sim U(0.1,100),c\sim U(0.1,10)$

\item Van Der Pol:
\begin{align}
    \dot{x}_1&=-x_2,\\
    \dot{x}_2&=x_1-\mu x_2(1-x_1^2),
\end{align}
with parameters $\mu\sim U(0.01,1.0)$.

\end{itemize}

\subsection{Model Implementation}

We use the implementation of FNO by \cite{kovachki2021neural}. For our experiments, we adopt an extended version of the standard FNO by integrating Adaptive Instance Normalization (AdaIN, \cite{huang2017arbitrary}) to handle conditioning on the parameters $Q$. These parameters are flattened and passed through a small MLP to compute normalization constants that modulate the intermediate features at various stages of the network,
\begin{equation}
    \text{AdaIN}(x,\alpha(Q),\beta(Q))=\alpha(Q)\frac{x-\mu(x)}{\sigma(x)}+\beta(Q),
\end{equation}
where $\mu(x)$ and $\sigma(x)$ are the mean and standard deviation of the layer, and $\alpha(Q),\beta(Q)$ are trainable MLPs that map the matrix $Q$ to scalars that determine scaling and shifting in the normalized layer.

\subsection{Performance Comparison}

We compared the FNO with DeepONet \cite{lu2019deeponet} in learning solutions to the Lyapunov PDE. We assessed the performance of these models by evaluating them on the testing set using the relative $L_1$ error, which computes the $L_1$ error between the predicted and the true Lyapunov function and divides by the $L_1$ norm of the true Lyapunov function. This is the default evaluation metric used in \cite{kovachki2021neural,herde2024poseidon}.

Table \ref{tab:placeholder} summarizes the $L_1$ errors on the test set. FNO achieved a lower error (0.0182) compared to DeepONet (0.6483), demonstrating its improved ability to approximate Lyapunov functions.

\begin{table}[H]
    \centering
    \begin{tabular}{c|c}
        FNO & DeepONet \\
        \hline
        0.0182 & 0.6483
    \end{tabular}
    \caption{$L_1$ errors of FNO and DeepONet on learning the Lyapunov PDE}
    \label{tab:placeholder}
\end{table}
\subsection{Visualization of Learned Functions}

Figure \ref{fig:placeholder} illustrates test cases across the systems. For each case, we plot the learned Lyapunov function against the ground-truth solution. The FNO reconstructions closely follow the true solutions, capturing its structure.

\section{Conclusion}

In this work, we introduced a framework for learning Lyapunov functions of nonlinear dynamical systems using FNOs By formulating the Lyapunov condition as a PDE and leveraging the universality of FNOs on Sobolev spaces, we established theoretical guarantees ensuring that Lyapunov operators can be approximated with high fidelity. Our analysis demonstrated regularity and continuity properties that justify the application of neural operator learning in this setting.
Through numerical experiments on nonlinear systems, we showed that FNOs achieve substantially lower approximation error compared to DeepONets.



A key challenge is extending this framework to higher-dimensional dynamical systems. In principle, the universality of FNOs extends naturally to $\Rb^n$, but practical implementation quickly becomes computationally prohibitive. Even for 3D problems, Fast Fourier Transforms (FFTs) scale as $O(N^3\log N)$ in time and $O(N^3)$ in memory, where $N$ is the resolution along each dimension \cite{van1992computational}. This growth imposes severe memory and runtime bottlenecks during training. 

Finally, combining learned Lyapunov functions with controller design remains a promising avenue. By embedding Lyapunov certificates into feedback synthesis, one may learn controllers with formal guarantees. This would bridge data-driven stability analysis with practical control implementation, further motivating the development of scalable and interpretable neural operator architectures. 
An early step in this direction was accomplished using diffusion models \cite{mukherjee2025manifold}.

\bibliography{ieee}
\bibliographystyle{abbrv}

\end{document}